\documentclass[12pt]{amsart}

\usepackage{amssymb,latexsym}

\usepackage{enumerate}

\makeatletter

\@namedef{subjclassname@2010}{

  \textup{2010} Mathematics Subject Classification}

\makeatother
\newtheorem{thm}{Theorem}[section]
\newtheorem{Con}[thm]{Conjecture}

\newtheorem{lem}[thm]{Lemma}
\newtheorem{pro}[thm]{Proposition}
\theoremstyle{definition}

\newtheorem{rem}[thm]{Remark}

\numberwithin{equation}{section}

\newcommand{\re}{\textup{Re}}
\newcommand{\im}{\textup{Im}}

\begin{document}

\baselineskip=17pt

\title[Large values of $L(1,\chi)$ for $k$-th order characters $\chi$]{Large values of $L(1,\chi)$ for $k$-th order characters $\chi$ and applications to character sums}

\author[Youness Lamzouri]{Youness Lamzouri}

\address{Department of Mathematics and Statistics,
York University,
4700 Keele Street,
Toronto, ON,
M3J1P3
Canada}

\email{lamzouri@mathstat.yorku.ca}

\date{}

\begin{abstract}  
For any given integer $k\geq 2$ we prove the existence of infinitely many $q$ and characters $ \chi\pmod q$ of order $k$, such that $|L(1,\chi)|\geq (e^{\gamma}+o(1))\log\log q$. We believe this bound to be best possible. When the order $k$ is even, we obtain similar results for $L(1,\chi)$ and $L(1,\chi\xi)$ where $\chi$ is restricted to even (or odd) characters of order $k$, and $\xi$ is a fixed quadratic character. As an application of these results, we exhibit large even order character sums, which are likely to be optimal. 

\end{abstract}


\thanks{The author is partially supported by a Discovery Grant from the Natural Sciences and Engineering Research Council of Canada.}

\maketitle

\section{Introduction}

Dirichlet characters of a fixed order appear naturally in many applications in number
theory. The quadratic characters have been extensively studied, due in large part to their connection to fundamental arithmetic objects including class numbers  and ranks of elliptic curves. By contrast, higher order characters have received considerably less attention up until very recently, when Granville and Soundararajan \cite{GrSo3} proved the remarkable result that the classical P\'olya-Vinogradov inequality can be improved for characters of a fixed odd order. Other notable work on higher order characters include the results of Baier and Young \cite{BaYo} on large sieve inequalities and moments of Dirichlet $L$-functions attached to cubic characters, which build on previous work by Heath-Brown \cite{HB} on cubic symbols; the large sieve inequalities for quartic characters by Gao and Zhao \cite{GaZa}; and the work of Blomer, Goldmakher and Louvel \cite{BGL} on large sieve inequalities and double Dirichlet series associated with certain higher order Hecke characters.

The connection between Dirichlet characters and class numbers  was discovered by Dirichlet in 1839, who established a formula that expresses the class number of a quadratic field $\mathbb{Q}(\sqrt{d})$ in terms of  $L(1,\chi_d)$, the value of the corresponding Dirichlet $L$-function at $1$, where $\chi_d:=\left(\frac{d}{\cdot}\right)$ is the Kronecker symbol.
 Motivated by Dirichlet's class number formula, Littlewood \cite{Li} studied how large can  $L(1, \chi)$ be, in terms of the conductor of $\chi$. Assuming  the Generalized Riemann Hypothesis (GRH), Littlewood proved that for any non-principal  primitive character $\chi \pmod q$, one has
 \begin{equation}\label{BoL1GRH}
|L(1,\chi)|\leq (2e^{\gamma}+o(1))\log\log q.
\end{equation}
On the other hand, under the same hypothesis,  Littlewood \cite{Li} showed that there exist infinitely many fundamental discriminants $d$ (both positive and negative) for which 
\begin{equation}\label{OmegaL1}
L(1,\chi_d)\geq (e^{\gamma}+o(1))\log\log |d|.
\end{equation}
This omega result  was later established unconditionally by  Chowla \cite{Ch}. 

To understand which of the bounds \eqref{BoL1GRH}  or \eqref{OmegaL1} is closer to the maximal values of $L(1, \chi_d)$,  Montgomery and Vaughan \cite{MoVa2} constructed a probabilistic random model for these values and made several conjectures on their distribution.  Most of these conjectures were subsequently proved by Granville and Soundararajan \cite{GrSo1}. Among their results, Granville and Soundararajan  obtained an asymptotic formula for the distribution function of $L(1,\chi_d)$, showing that the tail of this distribution is double exponentially decreasing.  In particular, their work  gives strong support to the conjecture that Chowla's omega result \eqref{OmegaL1} corresponds to the true nature of extreme values of $L(1,\chi_d)$. In \cite{GrSo2}, Granville and Soundararajan proved similar results for the distribution of the values $|L(1,\chi)|$ as $\chi$ varies over the  non-principal primitive characters modulo $q$ with $q\leq Q$.  Their results give very solid evidence for the following  widely believed conjecture.
\begin{Con}\label{ConjectureL1} Let $Q$ be large. Then 
$$
\max_{q\leq Q} \ \max_{\substack{\chi\neq\chi_0\pmod q\\ \chi  \text{ primitive}}} |L(1,\chi)|= (e^{\gamma}+o(1))\log\log Q.
$$
\end{Con}

Despite all the progress made on $L(1,\chi_d)$, very little is known on the values of $L(1,\chi)$ for higher order characters $\chi$.  The main difficulty is that, unlike the quadratic case where one is equipped with the powerful law of quadratic reciprocity,  higher reciprocity laws for $k$-th order symbols are not easy to apply to Dirichlet characters of order $k$. 

In this paper, we exhibit large values of $|L(1,\chi)|$ for $k$-th order characters $\chi$. We also apply our work to improve recent results of Goldmakher and the author \cite{GoLa}, and Bober \cite{Bo}, by obtaining lower bounds for even order character sums, which are likely to be optimal. 
Our first theorem extends Chowla's omega result \eqref{OmegaL1} to characters of any given order $k$. In view of Conjecture \ref{ConjectureL1} we believe our bound to be best possible. 

\begin{thm}\label{AllOrders}
Let $k\geq 2$ be fixed, and $Q$ be large. There exists a constant $c_k>0$ for which there are at least $Q\exp(-c_k\log Q/\log\log Q)$ primitive characters $\chi$ of order $k$ and conductor $q \leq Q$, such that 
$$|L(1,\chi)|\geq e^{\gamma} \log\log Q +O_k(1).$$
\end{thm}

Let $\chi \pmod q$ be a primitive character.  An important quantity attached to $\chi$ is 
$$M(\chi):=\max_{x}\left|\sum_{n\leq x}\chi(n)\right|.$$
This character sum has been extensively studied over the past century. The first non-trivial bound on $M(\chi)$, proved independently by P\'olya \cite{Po} and Vinogradov \cite{Vi} in 1918, asserts that 
$$M(\chi)\ll \sqrt{q} \log q.$$
This upper bound remains the strongest known outside of special cases. However, conditionally  on the GRH, Montgomery and Vaughan \cite{MoVa1} proved that 
\begin{equation}\label{MVBound}
M(\chi)\ll \sqrt{q} \log\log q.
\end{equation}
Recently, in a  groundbreaking paper \cite{GrSo3}, Granville and Soundararajan  improved both the P\'olya-Vinogradov inequality and the Montgomery-Vaughan GRH bound for characters of a given odd order. More precisely, they showed that if $g\geq 3$ is an odd integer, and $\chi\pmod q$ is a character of order $g$, then
\begin{equation}\label{GranvilleSoundPVI}
M(\chi)\ll  \begin{cases}  \displaystyle{\sqrt{q}(\log q)^{1-\delta_g}} & \text{ unconditionally},\\
 \displaystyle{\sqrt{q}(\log\log q)^{1-\delta_g}} & \text{ under GRH},
\end{cases}
\end{equation}
for some $\delta_g>0$. These bounds were subsequently sharpened by Goldmakher \cite{Go}.

The situation for even order characters is completely different. Indeed, an old result of Paley \cite{Pa} asserts the existence of an infinite family of quadratic characters $\chi \pmod q$ for which 
\begin{equation}\label{Paley}
M(\chi)\gg \sqrt{q}\log\log q,
\end{equation}
 showing that the Montgomery-Vaughan bound \eqref{MVBound} is sharp in this case. This was extended to characters of a given even order by Granville and Soundararajan \cite{GrSo3} under the assumption of GRH. Recently, Goldmakher and the author \cite{GoLa} obtained this result unconditionally. 

Granville and Soundararajan \cite{GrSo3} also refined the GRH bound of Montgomery and Vaughan \eqref{MVBound} for all characters $\chi \pmod q$. More specifically, assuming the GRH,  they showed that
\begin{equation}\label{GRHCharSum}
M(\chi)\leq  \begin{cases}  \displaystyle{\left(\frac{2e^{\gamma}}{\pi}+o(1)\right)\sqrt{q}\log\log q} & \text{ if } \chi \text{ is odd }(\text{that is }\chi(-1)=-1),\\
 \displaystyle{\left(\frac{2e^{\gamma}}{\pi\sqrt{3}}+o(1)\right)\sqrt{q}\log\log q} & \text{ if } \chi \text{ is even }(\text{that is }\chi(-1)=1).
\end{cases}
\end{equation}
Similarly to the case of $L(1, \chi)$,  Granville and Soundararajan \cite{GrSo3} conjecture that the GRH bounds \eqref{GRHCharSum} are off by a factor of $2$. Namely that, 
\begin{Con}\label{ConjectureCharSum} Let $q$ be large, and $\chi$ be a primitive character $\pmod q$. Then
\begin{equation}\label{ConjecturedCharSum}
M(\chi)\leq  \begin{cases}  \displaystyle{\left(\frac{e^{\gamma}}{\pi}+o(1)\right)\sqrt{q}\log\log q} & \text{ if } \chi \text{ is odd, }\\
 \displaystyle{\left(\frac{e^{\gamma}}{\pi\sqrt{3}}+o(1)\right)\sqrt{q}\log\log q} & \text{ if } \chi \text{ is even}.
\end{cases}
\end{equation}
\end{Con}
In a recent work,  Bober, Goldmakher, Granville and Koukoulopoulos \cite{BGGK} studied the distribution of large values of $M(\chi)$ as $\chi$ varies over non-principal primitive characters modulo $q$, where $q$ is a large prime. In particular, their results give strong support to Conjecture \ref{ConjectureCharSum}. 

In the other direction, Bateman and Chowla \cite{BaCh} improved Paley's result \eqref{Paley},  by establishing the existence of infinitely many $q$ and odd quadratic characters $\chi\pmod{q}$ such that 
$$
M(\chi)\geq \left(\frac{e^{\gamma}}{\pi}+o(1)\right)\sqrt{q}\log\log q.
$$
In view of Conjecture \ref{ConjectureCharSum}, this bound is likely to be best possible.  When $k\geq 4$ is an even integer, Goldmakher and the author \cite{GoLa} extended Paley's construction to characters of order $k$. More precisely, they showed that there are infinitely many $q$, and even characters $\chi\pmod q$ of order $k$ such that 
$$ M(\chi)\geq \left(\frac{1}{\pi\sqrt{p_k}}+o(1)\right)\sqrt{q}\log\log q,$$
where $p_k$ is the smallest prime such that $p_k\equiv k+1\pmod{2k}$. Since $p_k>k$,  this bound decreases as the order increases.  Bober \cite{Bo} subsequently obtained the same result using a different approach. As Bober notes that, this leaves open the possibility that for any $\epsilon>0$ there exist (large) even $k$ such that all
characters of order $k$ satisfy $M(\chi)\leq  \epsilon \sqrt{q}\log\log q$. We resolve this matter, by exhibiting values of $M(\chi)$ 
for even and odd characters $\chi$ of a fixed even order $k$, that are as large as the conjectured bounds \eqref{ConjecturedCharSum}. 

\begin{thm}\label{EvenOrdCharSum} Let $k\geq 2$ be a fixed even integer. There exist arbitrarily large $q$ and odd primitive characters $\chi \pmod q$ of order $k$, such that 
$$
M(\chi)\geq \left(\frac{e^{\gamma}}{\pi}+o(1)\right)\sqrt{q}\log\log q.
$$
Moreover,  there exist arbitrarily large $q$ and even primitive characters $\chi \pmod q$ of order $k$, such that 
$$
M(\chi)\geq \left(\frac{e^{\gamma}}{\pi\sqrt{3}}+o(1)\right)\sqrt{q}\log\log q.
$$
\end{thm}
In order to establish this result, we relate $M(\chi)$ to $L(1,\chi)$ via the following bounds, which are valid if the order of $\chi$ is even
\begin{equation}\label{BoundMChi}
M(\chi)\geq 
\begin{cases} \displaystyle{\frac{\sqrt{q}}{\pi}|L(1,\chi)|}& \text{ if } \chi \text{ is odd, and has even order},\\
\displaystyle{\frac{\sqrt{3q}}{2\pi}\left|L\left(1,\chi\left(\frac{\cdot}{3}\right)\right)\right|} & \text{ if } \chi \text{ is even, and has even order},\\
\end{cases}
\end{equation}
where $\left(\frac{\cdot}{3}\right)$ is the Legendre symbol modulo $3$. When $\chi$ is even, this bound is proved in Section 4 of \cite{BGGK}. When $\chi$ is odd, the corresponding bound follows from the pointwise estimate (see for example Theorem 9.21 of \cite{MoVa3})
$$\sum_{n\leq q/2}\chi(n)=(2-\chi(2))\frac{\tau(\chi)}{i\pi} \overline{L(1, \chi)}, $$
where $\tau(\chi)$ is the Gauss sum associated to $\chi$, which satisfies $|\tau(\chi)|=\sqrt{q}$. 

Now to obtain Theorem \ref{EvenOrdCharSum}, it remains to produce large values of $|L(1,\chi)|$ for odd primitive characters $\chi$ of even order $k$, as well as large values of $|L\left(1,\chi\left(\frac{\cdot}{3}\right)\right)|$ for even primitive characters $\chi$ of order $k$.   Unfortunately, our construction in Theorem \ref{AllOrders} does not allow us to restrict to odd or even characters, or to twist by  $\left(\frac{\cdot}{3}\right)$. Instead, we use a different construction based on twisting a family of quadratic characters by a single specific character of order $k$, so that these twists also have order $k$, since $k$ is even.  We prove 
\begin{thm}\label{EvenOrdersFixedSign}
Let $k\geq 2$ be a fixed even integer and $\delta\in \{1, -1\}$. Let $\xi$ be a primitive real character  of conductor $\ell$. Then, there are at least $Q^{1/3+o(1)}$ primitive characters $\chi$ of order $k$ and conductor $q\leq Q$ such that $\chi(-1)=\delta$ and
$$|L(1, \chi \xi )| \geq \frac{\varphi(\ell)}{\ell} e^{\gamma}\log\log Q +O_{\ell}(1).$$
\end{thm}
Using the bounds \eqref{BoundMChi}, we deduce Theorem \ref{EvenOrdCharSum} from Theorem \ref{EvenOrdersFixedSign} by taking $\xi$ to be the principal character if $\chi$ is odd, and $\xi= \left(\frac{\cdot}{3}\right)$ if $\chi$ is even.


\section{Large values of $|L(1, \chi)|$ for characters $\chi$ of order $k$: Proof of Theorem \ref{AllOrders}}
Our construction in Theorem \ref{AllOrders} relies on primitive characters $\chi$ of prime conductor $q$. In this case, if $\chi$ has order $k$ then $q\equiv 1 \pmod k$.  We first prove that there are exactly $\varphi(k)$ such characters modulo $q$. Here and throughout we let $\zeta_k:= \exp\left(\frac{2\pi i}{k}\right)$. 
\begin{lem}\label{Construction1}
Let $k\geq 2$ be an integer. For any prime $q\equiv 1\pmod k$, there are exactly $\varphi(k)$ primitive characters of order $k$ and conductor $q$. 
\end{lem}

\begin{proof} 
Let $g$ be a primitive root modulo $q$. Since any primitive character is completely determined by its value at $g$, we deduce that if $\chi$ is primitive and 
\begin{equation}\label{PropCharOrderk}
\chi(g)=(\zeta_k)^{\alpha}, \text{ for some } (\alpha, k)=1,
\end{equation}
then $\chi$ has order $k$. Moreover, any character $\chi$ of order $k$ and conductor $q$ has to satisfy \eqref{PropCharOrderk}. Finally, note that there are exactly $\varphi(k)$ such characters.
\end{proof}
For any prime $q\equiv 1 \pmod k$, we let $g_q$ be the smallest primitive root modulo $q$. We define $\psi_q$ to be the primitive character modulo $q$ such that 
\begin{equation}\label{DefCharRoot}
\psi_q(g_q)= \zeta_k.
\end{equation}
Note that $\psi_q$ has order $k$. Using these characters, we are going to construct a family of primitive characters of order $k$, which shall be used in the proof of Theorem \ref{AllOrders}.

\begin{lem}\label{character}
 Let $k\geq 2$ be an integer, and  $m=q_1q_2$ where $q_1\neq  q_2$ are primes such that $q_1\equiv q_2\equiv 1 \pmod k$. Then  
 $$
 \widetilde{\psi}_m:=\psi_{q_1}\overline{\psi_{q_2}}
 $$ is a primitive character of order $k$ and conductor $m$. 
\end{lem}
\begin{proof}
Since $q_1$ and $q_2$ are coprime, then $\widetilde{\psi}_m$ is primitive and has conductor $m$.  Moreover, $\psi_{q_1}$ and $\psi_{q_2}$ have order $k$ and hence the order of $\widetilde{\psi}_m$ divides $k$. Therefore, to show that $\widetilde{\psi}_m$ has order $k$, it suffices to find a integer $n$ such that 
$\widetilde{\psi}_m(n)=\zeta_{k}$. Now, let $g$ be the smallest primitive root modulo $q_1$, and $a$ be a solution to the following linear congruence
$$
aq_2\equiv g-1\pmod {q_1}.
$$
 Letting $n=aq_2+1$ and using \eqref{DefCharRoot} we derive
$$\widetilde{\psi}_m(n)=\psi_{q_1}(g)\overline{\psi_{q_2}(aq_2+1)}= \zeta_k,$$
as desired.
\end{proof}
Let $\mathcal{F}_k(Q)$ be the set of characters $\widetilde{\psi}_m$ indexed over the  integers $m=q_1q_2$, where $q_1, q_2$ are primes such that $\sqrt{Q}<q_1<q_2<2\sqrt{Q}$ and $q_1\equiv q_2\equiv 1\pmod k$. 
Then, it follows from the prime number theorem in arithmetic progressions that 
$$|\mathcal{F}_k(Q)|\asymp_k \frac{Q}{(\log Q)^2}.$$

In order to obtain large values of $L(1, \chi)$, a general strategy is to construct characters $\chi$ 
such that $\chi(p)=1$ for all the small primes $p$, typically up to the logarithm of the conductor of $\chi$. 
Using a judicious application of the pigeonhole principle we prove that there exist many characters $\widetilde{\psi}_m$ in $\mathcal{F}_k(Q)$ with this property.
\begin{lem}\label{Pigeonhole} Let $k\geq 2$ be a fixed integer. There exists a  constant $c_k>0$ (depending only on $k$)  for which there 
$$
\gg Q\exp\left(-c_k\frac{\log Q}{\log\log Q}\right)
$$ characters  $\widetilde{\psi}_m\in \mathcal{F}_k(Q)$ such that 
$\widetilde{\psi}_m(p)=1$ for all primes $p \leq \log Q.$
\end{lem}
\begin{proof} 
Let $p_j$ denote the $j$-th prime and $p_\ell$ be the largest prime below $\log Q$.  Note that $\ell=\pi(\log Q)\sim \log Q/\log\log Q$. Let 
$$
U_k= \{\zeta_k^a: 0\leq a\leq k-1\}
$$ be the set of $k$-th roots of unity, and $$\mathcal{A}= \{\mathbf{v}=(v_1, \dots, v_{\ell}) \text{ such that } v_j\in U_k \text{ for all } 1\leq j\leq \ell\}.$$
Note that $|\mathcal{A}|= k^{\ell}.$
For any $\mathbf{v}\in \mathcal{A}$ we define $S_Q(\mathbf{v})$ to be the set of primes $q$ in $(\sqrt{Q}, 2\sqrt{Q})$ such that $q\equiv 1 \bmod k$ and  $\psi_q(p_j)= v_j$  for all $1\leq j\leq \ell$. By the prime number theorem in arithmetic progressions we have 
$$
\sum_{\mathbf{v}\in \mathcal{A}}|S_Q(\mathbf{v})|=\pi\left(2\sqrt{Q}, k, 1\right)- \pi\left(\sqrt{Q}, k, 1\right)\sim \frac{2\sqrt{Q}}{\varphi(k)\log Q}.$$
Therefore, one has
$$ \max_{\mathbf{v}\in \mathcal{A}}|S_Q(\mathbf{v})|\gg \frac{\sqrt{Q}}{k^{\ell+1}\log Q}\gg \sqrt{Q}\exp\left(-\frac{c_k}{2}\frac{\log Q}{\log\log Q}\right),$$
for some positive constant $c_k>0$.
Let $\mathbf{v}_{\text{max}}$ be such that 
$$
\max_{\mathbf{v}\in \mathcal{A}}|S_Q(\mathbf{v})|= |S_Q(\mathbf{v}_{\text{max}})|.
$$
Note that if $q_1\neq q_2$ are both in $S_Q(\mathbf{v}_{\text{max}})$ then $\psi_{q_1}\overline{\psi_{q_2}}(p_j)=1$ for all $1\leq j\leq \ell.$ Finally, the number of characters $\widetilde{\psi}_m\in \mathcal{F}_k(Q)$ such that  $m=q_1q_2$ and $q_1, q_2 \in S_Q(\mathbf{v}_{\text{max}})$ equals
$$ \binom{|S_Q(\mathbf{v}_{\text{max}})|}{2}\gg Q\exp\left(-c_k\frac{\log Q}{\log\log Q}\right),$$
as desired.
\end{proof}

The second ingredient in the proof of Theorem \ref{AllOrders} is to approximate $L(1, \chi)$ by a short Euler product. Using zero density estimates together with the large sieve,  Granville and Soundararajan  \cite{GrSo1} proved that this can be done for almost all primitive characters  $\chi \pmod q$ with $q\leq Q$. More precisely, they established
 \begin{pro}[Proposition 2.2 of \cite{GrSo1}]\label{GranvilleSound}
Let $A>2$ be fixed. Then, for all but at most $Q^{2/A+o(1)}$ primitive characters $\chi \pmod q$ with $q\leq Q$ we have 
$$ 
L(1,\chi)=\prod_{p\leq (\log Q)^A}\left(1-\frac{\chi(p)}{p}\right)^{-1}\left(1+O\left(\frac{1}{\log\log Q}\right)\right).
$$
\end{pro}
Let  $z=(\log Q)^A$ for some $A>2$,  and $y\leq (\log Q)$ be a large real number. Then, note that
$$
\prod_{p\leq z}\left(1-\frac{\widetilde{\psi}_m(p)}{p}\right)^{-1}= \prod_{p\leq y}\left(1-\frac{\widetilde{\psi}_m(p)}{p}\right)^{-1}\exp\left(\sum_{y<p<z} \frac{\widetilde{\psi}_m(p)}{p}+O\left(\frac{1}{\sqrt{y}\log y}\right)\right).
$$
By Lemma \ref{Pigeonhole}, there are at least $Q^{1+o(1)}$ characters $\widetilde{\psi}_m\in \mathcal{F}_k(Q)$ for which the product  $\prod_{p\leq y}(1-\widetilde{\psi}_m(p)/p)^{-1}$ is as large as possible. The third and last ingredient in the proof of Theorem \ref{AllOrders} is the following proposition which gives an upper bound for the $2k$-th moment of $\sum_{y<p<z}\widetilde{\psi}_m(p)/p$ over the characters  $\widetilde{\psi}_m\in \mathcal{F}_k(Q)$, uniformly for $k$ in a large range. In particular, we shall later deduce that with very few exceptions in $ \mathcal{F}_k(Q)$, the prime sum $\sum_{y<p<z}\widetilde{\psi}_m(p)/p$ is small. 
\begin{pro}\label{LargeSieve} Let $k\geq 2$ be a fixed integer, and $k< y<(\log Q)^A$ be  a real number where $A>2$ is a constant. Put $z=(\log Q)^A$.  Then, for every positive integer $r\leq \log Q/(3Ak^2\log\log Q)$ we have 
$$\sum_{\widetilde{\psi}_m\in \mathcal{F}_k(Q)}\left|\sum_{y< p< z}\frac{\widetilde{\psi}_m(p)}{p}\right|^{2r}\ll_k 2^r r! Q\left(\sum_{y<p<z}\frac{1}{p^2}\right)^{r}+ Q^{1-1/(4k)}.$$
\end{pro}
To establish this result, we shall use the following large sieve inequality over $k$-th order characters,  which is due to Elliott \cite{El}. 

\begin{lem}[Lemma 33 of \cite{El}]\label{Elliott}  Let $k\geq 2$ be a fixed integer, and $\{\lambda_n\}_{n\geq 1}$ be a sequence of complex numbers. Then 
\begin{align*}
\sum_{\substack{p\leq Q\\ p\equiv 1\pmod k}} \ \ \sideset {}{^\star}\sum_{\chi\pmod p} &\left|\sum_{n\leq H}\lambda_n\chi(n)\right|^2
\ll_k Q\left|\sum_{\substack{m, n\leq H\\ mn^{k-1}=\mathfrak{a}^k}} \lambda_n\overline{\lambda_m}\right|+ H^{k}Q^{1-2/(k+1)}\left(\sum_{n\leq H}|\lambda_n|\right)^2,
\end{align*}
where the sum $\sideset {}{^\star}\sum_{\chi\pmod p}$ is confined to characters of order $k$ modulo $p$ for each prime modulus $p$, and $\mathfrak{a}$ is an algebraic integer of the cyclotomic field $K=\mathbb{Q}(\zeta_k)$. 
\end{lem}

\begin{rem}\label{algebraic}
If the sequence $\lambda_n$ is supported only on integers $n$ which are coprime to $k$ (that is $\lambda_n=0$ for all $(n, k)>1$), then one might replace the condition  $mn^{k-1}=\mathfrak{a}^k$ by $mn^{k-1}$ is an integral $k$-th power (i.e. one might take $\mathfrak{a}\in \mathbb{N}$). To see this, put $M= mn^{k-1}$ and let $M= p_1^{\alpha_1}\cdots p_{\ell}^{\alpha_{\ell}}$ be its prime factorization in $\mathbb{N}$. If $M=\mathfrak{a}^k$ for some alegbraic integer $\mathfrak{a}$ of $K$  then we must have $p_j^{\alpha_j}= \mathfrak{a}_j^k$ for some algebraic integers $\mathfrak{a}_1, \dots, \mathfrak{a}_{\ell}$ of $K$. Moreover, since $p_j\nmid k$, then $p_j$ is unramified in $K$, and hence we must have $k|\alpha_j$ for all $1\leq j\leq \ell$, which implies that $M$ is an integral $k$-th power.

\end{rem}

\begin{proof}[Proof of Proposition \ref{LargeSieve}]
First, observe that 
$$ \left(\sum_{y< p< z}\frac{\widetilde{\psi}_m(p)}{p}\right)^r= 
\sum_{y^r< n < z^r} \frac{b_r(n)\widetilde{\psi}_m(n)}{n},$$
where 
\begin{equation}\label{Combinatorics}
b_{r}(n):= \sum_{\substack{y<p_1, \dots, p_{r}<z\\ p_1\cdots p_{r}=n}}1.
\end{equation}
Note that $0\leq b_{r}(n)\leq r!$. Moreover, $b_{r}(n)=0$ unless $n=p_1^{\alpha_1}\cdots p_{s}^{\alpha_s}$ where  $y< p_1<p_2<\cdots<p_s< z$ are distinct primes and $\Omega(n)=\alpha_1+\cdots+\alpha_s=r$ (where $\Omega(n)$ is the number of prime divisors of $n$ counting multiplicities). In this case, we have 
\begin{equation}\label{CombFormula}
b_{r}(n)=\binom{r}{\alpha_1, \dots, \alpha_s}.
\end{equation}
Using this formula, one can easily deduce that if $n_1$ and $n_2$ are positive integers with $\Omega(n_1)=r_1$ and $\Omega(n_2)=r_2$ then
\begin{equation}\label{CombInequality}
b_{r_1+r_2}(n_1n_2)\leq  \binom{r_1+r_2}{r_1} b_{r_1}(n_1)b_{r_2}(n_2).
\end{equation}
Recall that for all $\widetilde{\psi}_m\in\mathcal{F}_k(Q)$, there exist primes $\sqrt{Q}<q_1<q_2< 2\sqrt{Q}$ such that $q_1\equiv q_2 \equiv 1\pmod k$ and $m=q_1q_2$. In this case we have $\widetilde{\psi}_m(n)=\psi_{q_1}(n)\overline{\psi_{q_2}}(n).$ Therefore, we have
\begin{equation}\label{expansion}
\sum_{\widetilde{\psi}_m\in \mathcal{F}_k(Q)}\Bigg|\sum_{y< p< z}\frac{\widetilde{\psi}_m(p)}{p}\Bigg|^{2r}\leq \sum_{\substack{q_1< 2\sqrt{Q}\\ q_1\equiv 1\bmod k}} \ \sum_{\substack{q_2< 2\sqrt{Q}\\ q_2\equiv 1\bmod k}}\left|\sum_{y^r< n< z^r} \frac{b_r(n)\overline{\psi_{q_1}(n)}}{n}\psi_{q_2}(n)\right|^2.
\end{equation}
We define $$\lambda_n=\frac{b_r(n)\overline{\psi_{q_1}(n)}}{n}.$$ 
Since $y>k$ then $b_r(n)=0$ if $(n,k)>1$. Therefore, it follows from Lemma \ref{Elliott} and Remark \ref{algebraic} that 
\begin{equation}\label{ElliotLargeS}
\begin{aligned}
&\sum_{\substack{q_2< 2\sqrt{Q}\\ q_2\equiv 1\bmod k}}\left|\sum_{y^r< n< z^r} \frac{b_r(n)\overline{\psi_{q_1}(n)}}{n}\psi_{q_2}(n)\right|^2\\
&\ll_k Q^{1/2}\left|\sum_{\substack{y^r< n_1, n_2<z^r \\ n_1n_2^{k-1}  \text{ is a } k \text{-th power}}} \frac{b_r(n_1)b_r(n_2)\overline{\psi_{q_1}(n_1)}\psi_{q_1}(n_2)}{n_1n_2}\right| + E_1(Q),\\
& \ll_k Q^{1/2}\sum_{\substack{y^r< n_1, n_2< z^r \\ n_1n_2^{k-1}\text{ is a } k \text{-th power}}} \frac{b_r(n_1)b_r(n_2)}{n_1n_2}+ E_1(Q),\\
\end{aligned}
\end{equation}
where 
\begin{equation}\label{error1}
E_1(Q)\ll_k z^{rk}Q^{1/2-1/(k+1)} \left(\sum_{y^r\leq n\leq z^r}\frac{b_r(n)}{n}\right)^2\ll z^{rk}Q^{1/2-1/(k+1)}\left(\sum_{y\leq p \leq z}\frac{1}{p}\right)^{2r} \ll Q^{1/2-1/(4k)}.
\end{equation}
Next, we bound the main term on the right hand side of \eqref{ElliotLargeS}. Let $n_1$ and $n_2$ be positive integers such that $\Omega(n_1)=\Omega(n_2)=r$ and 
put $d=(n_1,n_2)$. Also, put $n_1=dm_1$ and $n_2=dm_2$.  Since $n_1n_2^{k-1}$ is a $k$-th power then both $m_1$ and $m_2$ are $k$-th powers since $(m_1,m_2)=1.$ Let $m_1=\ell_1^k$ and $m_2=\ell_2^k$ , and put $s=\Omega(\ell_1)$. Since $\Omega(n_1)=\Omega(n_2)=r$, then  $\Omega(\ell_2)=s$ and $\Omega(d)=r-ks$. 
Therefore, by \eqref{CombInequality} we obtain
\begin{equation}\label{BoundCombCoef}
\begin{aligned}
b_r(n_1)b_r(n_2)
&\leq \binom{r}{ks}^2 b_{ks}(\ell_1^k)b_{ks}(\ell_2^k)b_{r-ks}(d)^2\\
&\leq \frac{(r!)^2}{(ks)!^2(r-ks)!} b_{ks}(\ell_1^k)b_{ks}(\ell_2^k)b_{r-ks}(d),\\
\end{aligned}
\end{equation}
since $b_{t}(n)\leq t!$ for any positive integers $t$ and $n$. 
Furthermore,  by \eqref{CombInequality} together with a simple inductive argument we derive
$$
b_{ks}(\ell_i^k)\leq \frac{(ks)!}{(s!)^k}b_{s}(\ell_i)^k\leq \frac{(ks)!}{s!}b_{s}(\ell_i) \text{ for } i=1,2.
$$
Inserting this estimate in \eqref{BoundCombCoef} yields 
$$
b_r(n_1)b_r(n_2)\leq \frac{(r!)^2}{(s!)^2(r-ks)!} b_{s}(\ell_1)b_{s}(\ell_2)b_{r-ks}(d)\leq r! \binom{r}{ks} b_{s}(\ell_1)b_{s}(\ell_2)b_{r-ks}(d),
$$
since $(ks)!\geq (2s)!\geq (s!)^2.$ Thus we deduce
\begin{equation}\label{BoundDiagonalLargeSie}
\begin{aligned}
\sum_{\substack{y^r<n_1, n_2< z^r \\ n_1n_2^{k-1} \text{ is a } k \text{-th power}}} \frac{b_r(n_1)b_r(n_2)}{n_1n_2}
&\leq  r!\sum_{0\leq s\leq r/k}\binom{r}{ks} \left(\sum_{d}\frac{b_{r-ks}(d)}{d^2}\right) \left(\sum_{\ell} \frac{b_{s}(\ell)}{\ell^k}\right)^2\\
&=r!\sum_{0\leq s\leq r/k}\binom{r}{sk} \left(\sum_{y<p<z}\frac{1}{p^2}\right)^{r-sk}\left(\sum_{y<p<z}\frac{1}{p^k}\right)^{2s},
\end{aligned}
\end{equation}
since 
$$
\sum_{n}\frac{b_t(n)}{n^{\alpha}}=\left(\sum_{y<p<z}\frac{1}{p^{\alpha}}\right)^t.
$$
Furthermore, since $k\geq 2$ then for any positive integer $n$, the Euclidean norm in $\mathbb{R}^n$ is larger than the $k$-norm. Therefore we have 
$$\left(\sum_{y<p<z}\frac{1}{p^k}\right)^{1/k}\leq \left(\sum_{y<p<z}\frac{1}{p^2}\right)^{1/2}.$$  Inserting this bound in \eqref{BoundDiagonalLargeSie} yields
 $$ \sum_{\substack{y^r< n_1, n_2<z^r \\ n_1n_2^{k-1} \text{ is a } k \text{-th power }}} \frac{b_r(n_1)b_r(n_2)}{n_1n_2}
 \leq 2^r r! \left(\sum_{y<p<z}\frac{1}{p^2}\right)^{r}.
 $$
Combining this bound with equations \eqref{expansion}, \eqref{ElliotLargeS} and \eqref{error1}  completes the proof.

\end{proof}

We are now ready to prove Theorem \ref{AllOrders}.

\begin{proof}[Proof of Theorem \ref{AllOrders}]
Let $z= (\log Q)^{5}$, and $2\leq y\leq (\log Q)^{3/2}$ be a real number to be chosen later. Then, by Proposition  \ref{GranvilleSound} it follows that for all but at most $Q^{1/2}$ primitive characters $\chi\pmod q$ with $q\leq Q$ we have 
\begin{equation}\label{FirstCond}
\begin{aligned}
L(1,\chi)&= \prod_{p\leq z} \left(1-\frac{\chi(p)}{p}\right)^{-1}\left(1+O\left(\frac{1}{\log\log Q}\right)\right)\\
&= \prod_{p\leq y} \left(1-\frac{\chi(p)}{p}\right)^{-1}\exp\left(\sum_{y<p<z}\frac{\chi(p)}{p}\right) \left(1+O\left(\frac{1}{\log\log Q}+\frac{1}{\sqrt{y}\log y}\right)\right),
\end{aligned}
\end{equation}
by the prime number theorem. 

Furthermore, taking $r=[\log Q/(15k^2\log\log Q)]$ in Proposition \ref{LargeSieve} we obtain that the number of characters $\widetilde{\psi}_m\in \mathcal{F}_k(Q)$ such that 
$$ \left|\sum_{y<p<z}\frac{\widetilde{\psi}_m(p)}{p}\right|>\frac{1}{\log\log Q}$$ is 
\begin{equation}\label{SecondCond}
\ll Q \left(\frac{4r(\log\log Q)^2}{y\log y}\right)^{r} \ll Q\left(\frac{\log Q\log \log Q}{3k^2 y\log y}\right)^r,
\end{equation}
since $\sum_{y<p}1/p^2\leq 2/y\log y$ by the prime number theorem. 

On the other hand, it follows from Lemma \ref{Pigeonhole} that for some constant  $c_k>0$, there are at least $Q\exp(-c_k\log Q/\log\log Q)$ characters $\widetilde{\psi}_m\in \mathcal{F}_k(Q)$ for which $\widetilde{\psi}_m(p)=1$ for all primes $p\leq \log Q$.   Choosing $y= b_k\log Q$ for some suitably large constant $b_k>0$ we deduce from \eqref{FirstCond} and \eqref{SecondCond} that there are at least $Q\exp(-2c_k\log Q/\log\log Q)$ characters  $\widetilde{\psi}_m\in \mathcal{F}_k(Q)$ for which $\widetilde{\psi}_m(p)=1$ for all primes $p\leq \log Q$,  and such that \eqref{FirstCond} holds and 
$$ \left|\sum_{y<p<z}\frac{\widetilde{\psi}_m(p)}{p}\right|\leq \frac{1}{\log\log Q}.$$ For these characters $\widetilde{\psi}_m$, we have by \eqref{FirstCond} that
\begin{align*}
L(1,\widetilde{\psi}_m)&=\prod_{p\leq \log Q} \left(1-\frac{1}{p}\right)^{-1}\exp\left(\sum_{\log Q<p<y}\frac{\widetilde{\psi}_m(p)}{p}\right) \left(1+O\left(\frac{1}{\log\log Q}\right)\right) \\
&= e^{\gamma}\log\log Q+ O_k(1),
\end{align*}
 as desired.
 \end{proof}
 
 
 \section{Large values of $|L(1,\chi\xi)|$ for even order characters $\chi$: Proof of Theorem \ref{EvenOrdersFixedSign}}

It follows from Lemma \ref{Pigeonhole} that there is a character $\psi$ of order $k$ and large conductor $m=q_1q_2$ where $q_1$ and $q_2$ are primes, such that  $\psi(p)=1$ for all primes $p\leq (\log m)/2$. In order to prove Theorem \ref{EvenOrdersFixedSign}, our construction involves the  characters $\psi\chi_d$,
where $d$ ranges over a certain family of fundamental discriminants.   To exhibit large values of $L(1, \psi\chi_d\xi)$, we first prove that for many of these fundamental discriminants $d$ we have $ \chi_d(p)=\xi(p)$ for all the ``small'' primes $p$. 

 \begin{pro}\label{FirstPRIMES}
 Let  $\delta=\pm 1$, $Q$ be large and $2\leq y\leq \log Q$ be a real number. Put  $P(y)=\prod_{p\leq y} p$. Let $\varepsilon(p)= \pm 1$ for each prime $p$. Then there are 
 $$\gg \frac{Q}{2^{\pi(y)}\log y}$$
 fundamental discriminants $d\equiv 1\pmod 4$ such that $0<\delta d\leq Q$,  $(d,P(y))=1$  and $\chi_d(p)=\varepsilon(p)$ for all primes $p\leq y$. 
\end{pro}
To establish this proposition, we first need to count the fundamental discriminants $d$ such that $0<\delta d\leq Q$, $d\equiv 1\pmod 4$ and $p\mid d \implies p>y$. To this end, we use the following standard estimate whose proof we include for completeness.
\begin{lem}\label{CountFundDiscrimi} 
Let $m$ be a positive integer, and $Q$ be a large real number. The number of fundamental discriminants $d$ such that $0<\delta d\leq Q$, $d\equiv 1\pmod 4$ and $(d, m)=1$ equals
$$ \frac{3}{\pi^2}Q\prod_{p|2m} \left(1+\frac{1}{p}\right)^{-1}+O\left(d(m)Q^{1/2}\right), $$
where $d(m)=\sum_{b|m}1$ is the divisor function.
\end{lem}
\begin{proof} We only prove the estimate when $\delta=1$, since the proof for the case $\delta=-1$ follows along similar lines. Let $\chi_{-4}$ be the non-principal real primitive  character modulo $4$, and $N_m(Q)$ be the number of fundamental discriminants $0<d\leq Q$ such $d\equiv 1\pmod 4$ and $(d, m)=1$. Then, writing $\mu^2(d)= \sum_{h^2\mid d}\mu(h)$ we obtain
\begin{equation}\label{CountFD}
\begin{aligned}
N_m(Q)
&=\sum_{\substack{0<d\leq Q\\ d\equiv 1 \pmod 4\\ (d, m)=1}}\mu^2(d)\\
&=\frac{1}{2}\sum_{\substack{0<d\leq Q\\ (d, 2m)=1}} \big(1+\chi_{-4}(d)\big)\sum_{h^2|d} \mu(h)\\
&=\frac{1}{2}\sum_{\substack{h\leq \sqrt{Q}\\ (h, 2m)=1}} \mu(h)
\sum_{\substack{b\leq Q/h^2\\ (b, 2m)=1}}\big(1+\chi_{-4}(b)\big).\\
\end{aligned}
\end{equation}
Now, observe that
$$\sum_{\substack{b\leq Q/h^2\\ (b, 2m)=1}}1= \sum_{a|2m}\mu(a)\sum_{\substack{b\leq Q/h^2 \\ a|b}} 1=\frac{Q}{h^2} \sum_{a|2m}\frac{\mu(a)}{a}+O\big(d(2m)\big)= \frac{Q}{h^2}\prod_{p|2m}\left(1-\frac{1}{p}\right)+O\big(d(2m)\big).$$
Similarly, one has 
$$ \sum_{\substack{b\leq Q/h^2\\ (b, 2m)=1}}\chi_{-4}(b)=  \sum_{a|2m}\mu(a)\sum_{\substack{b\leq Q/h^2 \\ a|b}}\chi_{-4}(b)=\sum_{a|2m}\mu(a)\chi_{-4}(a)\sum_{r\leq Q/(ah^2)}\chi_{-4}(r)=  O\big(d(2m)\big).$$
Combining these estimates with \eqref{CountFD} and using that $d(2m)\leq 2d(m)$, we deduce 
\begin{equation}\label{CountFD2}
N_m(Q)= \frac{Q}{2}\prod_{p|2m}\left(1-\frac{1}{p}\right)\sum_{\substack{h\leq \sqrt{Q}\\ (h, 2m)=1}} \frac{\mu(h)}{h^2}+O\left(d(m)Q^{1/2}\right).
\end{equation}
Finally note that
$$ 
\sum_{\substack{h\leq \sqrt{Q}\\ (h, 2m)=1}} \frac{\mu(h)}{h^2}= \sum_{\substack{h\geq 1\\ (h, 2m)=1}} \frac{\mu(h)}{h^2}+O\big(Q^{-1/2}\big)=\frac{6}{\pi^2}\prod_{p|2m}\left(1-\frac{1}{p^2}\right)^{-1}+O\big(Q^{-1/2}\big).
$$
Inserting this estimate in \eqref{CountFD2} completes the proof.

\end{proof}
The second ingredient in the proof of Proposition \ref{FirstPRIMES} is the following bound on character sums, which is a slight variation of Lemma 4.1  of Granville and Soundararajan  \cite{GrSo1}. Here and throughout, the sum  $\sideset {}{^\flat}\sum_{d}$ is confined to fundamental discriminants $d$. 
\begin{lem}\label{BoundNonSquare}
Let $\delta=\pm 1$. Let $m$ be positive integer, and $n\geq 2$ be an integer, not a perfect square. Then we have
$$
 \sideset {}{^\flat}\sum_{\substack{0<\delta d \leq Q\\ d\equiv 1\pmod 4\\ (d, m)=1}} \chi_d(n)\ll d(m) Q^{1/2} n^{1/4}(\log n)^{1/2}.
$$

 \end{lem}

 \begin{proof}
We only prove the estimate when $\delta=1$, since the proof for the case $\delta=-1$ follows similarly. Writing $\mu^2(d)= \sum_{h^2\mid d}\mu(h)$ we get
\begin{equation}\label{CHARSum}
\begin{aligned}
\sideset {}{^\flat}\sum_{\substack{0< d \leq Q\\ d\equiv 1\pmod 4\\ (d, m)=1}} \chi_d(n)
&= \sum_{\substack{d\leq Q\\ d\equiv 1 \pmod 4\\ (d, m)=1}}\mu^2(d)\chi_d(n)\\
&=\frac{1}{2}
 \sum_{\xi\pmod 4}\sum_{\substack{h\leq \sqrt{Q}\\ (h, m)=1}}\mu(h)\sum_{\substack{d\leq Q\\ h^2|d\\ (d, m)=1}}\xi(d)\left(\frac{d}{n}\right)\\
 &= \frac{1}{2}
 \sum_{\xi\pmod 4}\sum_{\substack{h\leq \sqrt{Q}\\ (h, m)=1}}\mu(h)\xi(h^2)\left(\frac{h^2}{n}\right)\sum_{\substack{b\leq Q/h^2 \\ (b, m)=1}}\xi(b)\left(\frac{b}{n}\right)\\
 &= \frac{1}{2}
 \sum_{\xi\pmod 4}\sum_{\substack{h\leq \sqrt{Q}\\ (h, m)=1}}\mu(h)\xi(h^2)\left(\frac{h^2}{n}\right)\sum_{a|m}\mu(a)\sum_{\substack{b\leq Q/h^2 \\ a|b}}\xi(b)\left(\frac{b}{n}\right)\\
 &= \frac{1}{2}
 \sum_{\xi\pmod 4}\sum_{\substack{h\leq \sqrt{Q}\\ (h, m)=1}}\mu(h)\xi(h^2)\left(\frac{h^2}{n}\right)\sum_{a|m}\mu(a)\xi(a)\left(\frac{a}{n}\right)\sum_{\ell\leq Q/(ah^2)}\xi(\ell)\left(\frac{\ell}{n}\right).
 \end{aligned}
 \end{equation}
Now, since $\xi(\cdot)\left(\frac{\cdot}{n}\right)$ is a non-principal character of conductor at most $4n$, then by the P\'olya-Vinogradov inequality, we have 
$$ \sum_{\ell\leq Q/(a h^2)}\xi(\ell)\left(\frac{\ell}{n}\right) \ll \sqrt{n}\log n.$$
We use this bound in \eqref{CHARSum} if $h\leq Q^{1/2}n^{-1/4}(\log n)^{-1/2}$, and the trivial bound $Q/h^2$ if $h> Q^{1/2}n^{-1/4}(\log n)^{-1/2}$, in order to get
\begin{align*}
\sideset {}{^\flat}\sum_{\substack{0<d \leq Q\\ d\equiv 1\pmod 4\\ (d, m)=1}} \chi_d(n)
&\ll d(m)Q^{1/2}n^{1/4}(\log n)^{1/2}+ d(m)Q\sum_{h>Q^{1/2}n^{-1/4}(\log n)^{-1/2}}\frac{1}{h^2}\\
&\ll d(m)Q^{1/2}n^{1/4}(\log n)^{1/2},
\end{align*}
as desired.
 \end{proof}

\begin{proof}[Proof of Proposition \ref{FirstPRIMES}] Let $L=\pi(y)$. Note that if $(d, P(y))=1$ and $p$ is a prime $\leq y$, then $1+\chi_d(p)\varepsilon(p)=2$ if $\chi_d(p)=\varepsilon(p)$ and equals $0$ otherwise. Therefore, the number of fundamental discriminants $d\equiv 1 \pmod 4$ such that $0<\delta d\leq Q$, $(d,P(y))=1$  and $\chi_d(p)=\varepsilon(p)$ for all primes $p\leq y$ equals
\begin{align*}
& \frac{1}{2^L}\sideset {}{^\flat}\sum_{\substack{0<\delta d\leq Q\\ d\equiv 1 \pmod 4\\ (d, P(y))=1}} \ \prod_{p\leq y}\left(1+\chi_d(p)\varepsilon(p)\right)\\
=& \frac{1}{2^L}\sideset {}{^\flat}\sum_{\substack{0< \delta d\leq Q \\ d\equiv 1 \pmod 4\\ (d, P(y))=1}} \left(1+\sum_{r=1}^L \sum_{p_1<p_2<\cdots<p_r\leq y} \varepsilon(p_1)\cdots \varepsilon(p_r) \chi_d(p_1\cdots p_r)\right)\\
=&  \frac{1}{2^L}\sideset {}{^\flat}\sum_{\substack{0<\delta d\leq Q\\ d\equiv 1 \pmod 4\\ (d, P(y))=1}} 1+ \frac{1}{2^L}\sum_{r=1}^L \sum_{p_1<p_2<\cdots<p_r\leq y} \varepsilon(p_1)\cdots \varepsilon(p_r) \sideset {}{^\flat}\sum_{\substack{0<\delta d\leq Q \\ d\equiv 1 \pmod 4\\ (d, P(y))=1}}\chi_d(p_1\cdots p_r)\\
=& \frac{1}{2^L}\sideset {}{^\flat}\sum_{\substack{0<\delta d\leq Q \\ d\equiv 1 \pmod 4\\ (d, P(y))=1}} 1 +O\left(Q^{1/2} e^{y/3}\right),
\end{align*}
by Lemma \ref{BoundNonSquare} together with the fact that $\prod_{p\leq y}p=e^{y(1+o(1))}$ which follows from the prime number theorem. Finally, using Lemma \ref{CountFundDiscrimi} we get
$$ \sideset {}{^\flat}\sum_{\substack{0<d\leq Q \\ d\equiv 1 \pmod 4\\ (d, P(y))=1}} 1 \gg \frac{Q}{\log y},$$
which completes the proof.
\end{proof}
We also need an $L^{2k}$ bound for the prime sum $\sum_{y<p<z} (\psi\xi\chi_d)(p)/p$, similar to Proposition \ref{LargeSieve}. To this end we establish the following lemma. 

\begin{lem}\label{LargeSieveQuad}
Let $\{a(p)\}_{p \text{ prime}}$ be a sequence of complex numbers such that $|a(p)|\leq 1$. Let $A\geq 1$ be fixed and $2\leq y< (\log Q)^A $ be a real number. Put $z=(\log Q)^A$.  Then, for any positive integer $r\leq \log Q/(6A\log \log Q)$ we have 
$$ 
\frac{1}{Q} \sideset {}{^\flat}\sum_{\substack{|d| \leq Q\\ d\equiv 1 \pmod 4}}\left(\sum_{y<p<z}\frac{a(p)\chi_d(p)}{p}\right)^{2r}\ll \frac{(2r)!}{r!}\left(\sum_{y<p<z}\frac{1}{p^2}\right)^r + Q^{-1/3}.
$$

\end{lem}

\begin{proof}
First, we extend the sequence $\{a(p)\}_p$ multiplicatively to all positive integers $n>1$ by setting
$$ a(n)=a(p_1)^{\alpha_1}\cdots a(p_j)^{\alpha_j}, \text{ if } n=p_1^{\alpha_1}\cdots p_j^{\alpha_j}.$$
Then, we have 
$$
\left(\sum_{y<p<z}\frac{a(p)\chi_d(p)}{p}\right)^{2r}=\sum_{y^{2r}<n<z^{2r}} \frac{a(n)\chi_d(n)b_{2r}(n)}{n},
$$
where the coefficient $b_{2r}(n)$ is defined in \eqref{Combinatorics}. 
Therefore, using Lemma \ref{BoundNonSquare} we get
\begin{equation}\label{StandardLargeSieve}
\begin{aligned}
\sideset {}{^\flat}\sum_{\substack{|d| \leq Q\\ d\equiv 1 \pmod 4}}\left(\sum_{y<p<z}\frac{a(p)\chi_d(p)}{p}\right)^{2k}
&= \sum_{y^{2r}<n<z^{2r}} \frac{a(n)b_{2r}(n)}{n}\sideset {}{^\flat}\sum_{\substack{|d| \leq Q\\ d\equiv 1 \pmod 4}}\chi_d(n)\\
&\ll Q \sum_{y^r<m<z^r} \frac{b_{2r}(m^2)}{m^2} + Q^{1/2} \sum_{y^{2r}<n<z^{2r}} \frac{b_{2r}(n)}{n^{1/2}}\\
&\ll Q \sum_{y^r<m<z^r} \frac{b_{2r}(m^2)}{m^2} +Q^{1/2} \left(\sum_{y<p<z}\frac{1}{\sqrt{p}}\right)^{2r}.
\end{aligned}
\end{equation}
Furthermore, by \eqref{CombInequality} together with the fact that $b_r(m)\leq r!$ we obtain
$$ 
\sum_{y^r<m<z^r} \frac{b_{2r}(m^2)}{m^2}  \leq \binom{2r}{r} \sum_{y^r<m<z^r} \frac{b_{r}(m)^2}{m^2} \leq \frac{(2r)!}{r!} \sum_{y^r<m<z^r} \frac{b_{r}(m)}{m^2}= \frac{(2r)!}{r!}\left(\sum_{y<p<z}\frac{1}{p^2}\right)^r.
$$
Inserting this bound in \eqref{StandardLargeSieve} and using that 
$\sum_{y<p<z}1/\sqrt{p}\ll \sqrt{z}/\log z,
$
completes the proof.
\end{proof}
We now prove Theorem \ref{EvenOrdersFixedSign}. 

\begin{proof}[Proof of Theorem \ref{EvenOrdersFixedSign}]
First, by Lemma \ref{Pigeonhole} there exist prime numbers $q_1, q_2$ such that $Q^{1/3}<q_1<q_2<2Q^{1/3}$ and a character $\psi$ of order $k$ and conductor $q_1q_2$ such that $\psi(p)=1$ for all primes $p\leq 2(\log Q)/3$. 

Let $\varepsilon=\psi(-1)$, and $\ell<y\leq 2(\log Q)/3$ be a real number to be chosen later. We consider the family of characters $\{\psi \chi_d\}$,
where $d$ ranges over the fundamental discriminants $d\equiv 1\pmod 4$ such that $(d, P(y))=1$ and  $0<\varepsilon \delta d<Q^{1/3}$. Since $\ell, d,$ and $q_1q_2$ are pairwise coprime then $\psi\xi \chi_d$ is primitive and $\psi\chi_d$ is a primitive character of order $k$ and conductor $  |d|q_1q_2\ll Q$. Moreover, note that $\psi \chi_d(-1)= \delta$. 

Let $z=(\log Q)^{10}$. Then, it follows from Proposition \ref{GranvilleSound} that for all but at most $Q^{1/4}$ fundamental discriminants $d$ with $0<\varepsilon \delta d<Q^{1/3}$ we have 
\begin{equation}\label{TruncSpecialFamily}
L(1, \psi\xi \chi_d)
=\prod_{p\leq z}\left(1-\frac{\psi(p)\xi(p) \chi_d(p)}{p}\right)^{-1}\left(1+O\left(\frac{1}{\log\log Q}\right)\right).
\end{equation}
Furthermore, by Proposition \ref{FirstPRIMES} there are at least $\gg_{\ell}Q^{1/3}/(2^{\pi(y)}\log y)$ fundamental discriminants $d\equiv 1 \pmod 4$ such that $0<\varepsilon \delta d<Q^{1/3}$, $(d, P(y))=1$,  and $\chi_d(p)=\xi(p)$ for all primes $p\leq y$, such that $p\nmid \ell $.

Moreover, taking $r=[\log Q/(60\log\log Q)]$ in Lemma \ref{LargeSieveQuad} we obtain that the number of fundamental discriminants $d$ with $0<\varepsilon \delta d<Q^{1/3}$ such that 
$$ \left|\sum_{\log Q<p<z}\frac{\psi(p)\xi(p)\chi_d(p)}{p}\right|>\frac{1}{\log\log Q},$$ is 
$$
\ll Q^{1/3} \left(\frac{4r\log\log Q}{\log Q}\right)^{r} \ll Q^{1/3}\exp\left(-\frac{\log Q}{30\log \log Q}\right),
$$
since $\sum_{\log Q<p}1/p^2\leq 2/(\log Q\log\log Q)$ by the prime number theorem. 

Thus, choosing $y= (\log Q)/50$, we deduce that there are 
$$
\gg_{\ell}Q^{1/3}\exp\left(-\frac{\log Q}{50\log\log Q}\right)
$$
 fundamental discriminants $d\equiv 1\pmod 4$ such that $0<\varepsilon \delta d<Q^{1/3}$, $(d, P(y))=1$,  the asymptotic formula \eqref{TruncSpecialFamily} holds,  $\chi_d(p)=\xi(p)$ for all primes $p\leq y$ with $p\nmid \ell $, and 
$$ \left|\sum_{\log Q<p<z}\frac{\psi(p)\xi(p)\chi_d(p)}{p}\right|\leq \frac{1}{\log\log Q}.$$
For these $d$, we have by \eqref{TruncSpecialFamily} that 
\begin{align*}
L(1, \psi\xi \chi_d)
&=\prod_{\substack{p\leq y\\ p\nmid \ell}}\left(1-\frac{1}{p}\right)^{-1}\exp\left(\sum_{y<p\leq \log Q}\frac{\psi(p)\xi(p) \chi_d(p)}{p}\right)\left(1+O\left(\frac{1}{\log\log Q}\right)\right)\\
&= \frac{\varphi(\ell)}{\ell} e^{\gamma}\log\log Q +O_{\ell}(1),
\end{align*}
as desired. 

\end{proof}


\begin{thebibliography}{DDDD}

\bibitem[1] {BaYo} S. Baier and M. P. Young, 
\emph{Mean values with cubic characters.}
J. Number Theory 130 (2010), no. 4, 879--903. 

\bibitem[2] {BaCh} P. T. Bateman and S. Chowla, 
\emph{Averages of character sums.}
Proc. Amer. Math. Soc. 1 (1950), 781--787.

\bibitem[3] {BGL} V. Blomer, L. Goldmakher and B. Louvel, 
\emph{$L$-functions with $n$-th-order twists.}
Int. Math. Res. Not. IMRN 2014, no. 7, 1925--1955.

\bibitem[4] {Bo} J. Bober, 
\emph{Averages of character sums.}
Preprint (2014), 12 pages. arXiv:1409.1840.

\bibitem[5] {BGGK} J. Bober,  L. Goldmakher, A. Granville and D. Koukoulopoulos,
\emph{The frequency and the structure of large character sums.}
Preprint (2014), 58 pages. arXiv:1410.8189.

\bibitem[6] {Ch} S. Chowla, 
\emph{Improvement of a theorem of Linnik and Walfisz.}
 Proc. London Math. Soc. 50 (1949), 423--429.
 
\bibitem[7] {El} P. D. T. A. Elliott, 
\emph{On the mean value of $f(p)$.}
Proc. London Math. Soc. (3) 21 (1970) 28--96.

\bibitem[8] {GaZa} P. Gao and L. Zhao,  
\emph{Large sieve inequalities for quartic characters.}
Q. J. Math. 63 (2012), no. 4, 891--917.

\bibitem[9] {Go} L. Goldmakher,
\emph{Multiplicative mimicry and improvements to the P\'olya-Vinogradov inequality,.} 
Algebra and Number Theory 6 (2012), no. 1, 123--163, 

\bibitem[10] {GoLa} L. Goldmakher and Y. Lamzouri,
\emph{Large even order character sums.} 
Proc. Amer. Math. Soc. 142 (2014), no. 8, 2609--2614. 

\bibitem[11] {GrSo1} A. Granville and K. Soundararajan,
\emph{The distribution of values of $L(1, \chi_d)$.} 
Geom. Funct. Anal. 13 (2003), no. 5, 992--1028. 

\bibitem[12] {GrSo2} A. Granville and K. Soundararajan,
\emph{Extreme values of $|\zeta(1+it)|$.} 
The Riemann zeta function and related themes: papers in honour of Professor K. Ramachandra, 65--80, Ramanujan Math. Soc. Lect. Notes Ser., 2, Ramanujan Math. Soc., Mysore, 2006.

\bibitem[13] {GrSo3} A. Granville and K. Soundararajan,
\emph{Large character sums: pretentious characters and the P\'olya-Vinogradov theorem.} 
J. Amer. Math. Soc. 20 (2007), no. 2, 357--384.

\bibitem[14] {HB} D. R. Heath-Brown,
\emph{Kummer's conjecture for cubic Gauss sums.} 
Israel J. Math. 120 (2000), part A, 97--124. 

\bibitem[15]{Li}  J. E. Littlewood,
\emph{On the class number of the corpus $P(\sqrt{-k})$.} 
Proc. London Math. Soc. 27 (1928), 358--372.

\bibitem[16]{MoVa1}  H. L. Montgomery and R. C. Vaughan,
\emph{Exponential sums with multiplicative coefficients.} 
Invent. Math. 43 (1977), no. 1, 69--82.

\bibitem[17]{MoVa2}  H. L. Montgomery and R. C. Vaughan,
\emph{Extreme values of Dirichlet L-functions at $1$.} 
Number theory in progress, Vol. 2 (Zakopane-Ko\'scielisko, 1997), 1039--1052, de Gruyter, Berlin, 1999.

\bibitem[18]{MoVa3}  H. L. Montgomery and R. C. Vaughan,
\emph{Multiplicative number theory. I. Classical theory.} 
Cambridge Studies in Advanced Mathematics, vol. 97, Cambridge University Press, Cambridge, 2007.

\bibitem[19]{Pa}  R. E. A. C. Paley,
\emph{A theorem on characters.} 
J. London Math. Soc. 7 (1932), 28--32.

\bibitem[20]{Po} G. P\'olya,
\emph{Uber die Verteilung der quadratischen Reste und Nichtreste.} 
G\"ottingen Nachrichten (1918), 21--29.

\bibitem[21]{Vi}  I. M. Vinogradov,
\emph{Uber die Verteilung der quadratischen Reste und Nichtreste.} 
J. Soc. Phys. Math. Univ. Permi 2 (1919), 1--14.

\end{thebibliography}
\end{document}